\documentclass[12pt]{article}
\usepackage{amsthm, amsmath, amsfonts, url}

\theoremstyle{theorem}
\newtheorem{theorem}{Theorem}

\theoremstyle{definition}

\newtheorem*{note}{Note}
\newtheorem*{lemma}{Lemma}

\begin{document}
\title{Fibonacci Plays Billiards} 
\markright{Fibonacci Plays Billiards}
\author{Elwyn Berlekamp and Richard Guy} 
\date{2003-02-07}
\maketitle 

\begin{abstract}
A \textit{chain} is an ordering of the integers 1 to $n$ such that adjacent pairs have sums of a
particular form, such as squares, cubes, triangular numbers, pentagonal numbers, or Fibonacci
numbers.  For example  4 1 2 3 5  form a Fibonacci chain while 1 2 8 7 3 12 9 6 4 11 10 5 form a
triangular chain. Since 1 + 5 is also triangular, this latter forms a triangular \textit{necklace}.
A search for such chains and necklaces can be facilitated by the use of paths of billiard balls on
a rectangular or other polygonal billiard table.
\end{abstract}

At the July, 2002 Combinatorial Games Conference in
Edmonton we found Yoshiyuki Kotani looking for
values of $n$ which would enable him to arrange
the numbers 1 to $n$ in a chain so that adjacent
links summed to a perfect cube.  Part of such a
chain might be
$$\ldots\quad61\quad3\quad5\quad22\quad42\quad\ldots$$

He had seen the corresponding problem asked for squares.
Later Ed Pegg informed us that this latter problem, with
squares and with $n=15$, was proposed by Bernardo
Recaman Santos, of Colombia, at the 2000 World Puzzle
Championship.  More recently this has appeared as Puzzle 30
in \cite{Y}.

\begin{figure}[h]    
\begin{center}
\mbox{$(16\!\rightarrow)9\!\rightarrow\!7\!\rightarrow\!2\!\leftarrow\!
14\!\rightarrow\!11\!\rightarrow\!5\!\rightarrow\!4\!\leftarrow\!
12\!\leftarrow\!13\!\rightarrow\!3\!\leftarrow\!6\!\leftarrow\!
10\!\leftarrow\!15\!\rightarrow\!1\!\leftarrow\!8(\!\leftarrow\!17)$}
\end{center}
\caption{Solution(s) to Recaman's problem for $n=15$, 16, 17.}
\label{fig:RecamanProblem}
\end{figure}

This inspired Joe Kisenwether to ask for the numbers
1 to 32 to be arranged as a necklace whose neighboring beads
add to squares (Figure \ref{fig:necklace1}).

\begin{figure}[h]    
\begin{center}
\begin{tabular}{c@{\hspace{4pt}}c@{\hspace{4pt}}c
 @{\hspace{4pt}}c@{\hspace{4pt}}c@{\hspace{4pt}}c
 @{\hspace{4pt}}c@{\hspace{4pt}}c@{\hspace{4pt}}c}
 4 & 21 & 28 &  8 & 1  & 15 & 10 & 26 & 23  \\
32 &    &    &    &    &    &    &    &  2  \\
17 &    &    &    &    &    &    &    & 14 \\
19 &    &    &    &    &    &    &    & 22 \\
30 &    &    &    &    &    &    &    & 27 \\
 6 &    &    &    &    &    &    &    &  9 \\
 3 &    &    &    &    &    &    &    & 16 \\
13 &    &    &    &    &    &    &    & 20 \\
12 & 24 & 25 & 11 &  5 & 31 & 18 &  7 & 29 \\
\end{tabular}
\end{center}
\caption{A necklace with adjacent pairs of beads adding
to squares.}
\label{fig:necklace1}
\end{figure}

The extension to cubes was suggested by Nob Yoshigahara.
The least $n$ for such a chain or necklace may be greater
than 300.  But it seems certain that such chains and
necklaces can be found for all sufficiently large $n$,
and for any other powers or polynomials, e.g., figurate
numbers of various kinds; see Figure~\ref{fig:necklace2}.

\begin{figure}[h]       
\begin{center}
\begin{tabular}{c@{\hspace{3pt}}c@{\hspace{3pt}}c
 @{\hspace{3pt}}c@{\hspace{3pt}}c@{\hspace{0pt}}c
 @{\hspace{0pt}}c@{\hspace{2pt}}c@{\hspace{2pt}}c}
& & & & 3 & & & & \\
& & & 7 & & 12 & & & \\
& & 8 & & & & 9 & & \\
&  2 &  & & & & & 6 & \\
1 & & 5 & & 10 & & 11 & & 4 \\
\end{tabular}
\end{center}
\caption{A necklace with adjacent pairs of beads adding to
  triangular numbers.}
\label{fig:necklace2}
\end{figure}

So we asked about more rapidly divergent sequences.
For powers of 2, it is not possible to connect chains
of odd numbers to chains of even numbers, and there
are similar difficulties with powers of larger numbers.

However, the corresponding problem with neighbors
summing to Fibonacci numbers, $F_0=0$, $F_1=1$,
$F_{k+1}= F_k + F_{k-1}$, has a better balanced solution.

We can draw a graph with the numbers 1 to $n$ as vertices
and edges joining pairs whose sum is a Fibonacci number:
for $n=11$, this is Figure \ref{fig:FibGraph}.
\begin{figure}[h]
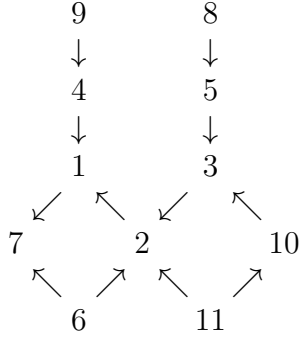
       
\begin{center}
\begin{tabular}{c@{\hspace{3pt}}c@{\hspace{3pt}}c@{\hspace{3pt}}
 c@{\hspace{3pt}}c@{\hspace{3pt}}c@{\hspace{2pt}}c@{\hspace{2pt}}
 c@{\hspace{2pt}}c}
& & 9 & & & & 8 & & \\
& & $\downarrow$ & & & & $\downarrow$ & & \\
& & 4 & & & & 5 & & \\
& & $\downarrow$ & & & & $\downarrow$ & & \\
& & 1 & & & & 3 & & \\
& $\swarrow$ & & $\nwarrow$ & & $\swarrow$ & & $\nwarrow$ & \\
7 & & & & 2 & & & & 10
\\
& $\nwarrow$ & & $\nearrow$ & & $\nwarrow$ & & $\nearrow$ & \\
& & 6 & & & & 11
\end{tabular}
\end{center}
\caption{Graph whose adjacencies are Fibonacci sums}
\label{fig:FibGraph}
\end{figure}
The arrows are drawn from the larger to the smaller number
to emphasize that the larger number is not part of the graph
unless the smaller is already present.  From the graph we
can read off 1 2; 1 2 3; 4 1 2 3; 4 1 2 3 5; 4 1 7 6 2 3 5;
4 1 7 6 2 3 5 8; 9 4 1 7 6 2 3 5 8 and 9 4 1 7 6 2 11 10 3 5 8.
We can also verify that 6 and 10 can't be included in a chain
unless some larger number is also present (in the former case
4, 5 and 6 are monovalent vertices and all three can't be ends
of the chain; in the latter case, 8, 9 and 10).  Evidently
the Law of Small Numbers is at work.  Six and ten are the
only numbers which are not powers of primes.  Is there some
connexion with projective planes?  No, but the Law of Small
Numbers is indeed at work, but the villains are 9 and 11.

\begin{theorem} \label{th:1}
There is a chain formed with the numbers 1 to
$n$ with each adjacent pair adding to a Fibonacci number, just
if $n=9$, 11, or $F_k$ or $F_k - 1$, where $F_k$ is a Fibonacci
number with $k\geq4$.  The chain is essentially unique.
\end{theorem}

\begin{proof}
For $n\leq11$ ($k=4$, 5, 6) this follows from
Figure \ref{fig:necklace1}.  If $k=7$, then $12=F_7-1$ can be appended to the
11-chain, forming a 4-circuit; also, $F_7=13$ can be appended
at the other end, as shown in Figure \ref{fig:ball+chain}.

\begin{figure}[h]
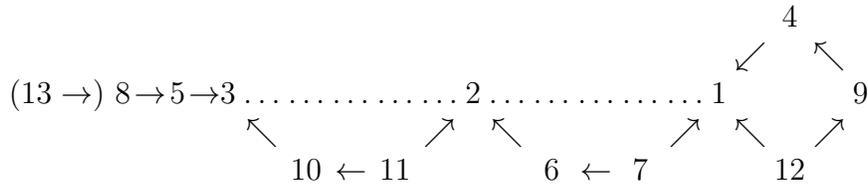
     
\begin{center}
\begin{tabular}{r@{\hspace{2pt}}c@{\hspace{3pt}}c@{\hspace{3pt}}c
  @{\hspace{3pt}}c@{\hspace{3pt}}c@{\hspace{3pt}}c@{\hspace{3pt}}c
  @{\hspace{3pt}}c@{\hspace{3pt}}c@{\hspace{3pt}}c@{\hspace{3pt}}c
  @{\hspace{3pt}}c@{\hspace{3pt}}c@{\hspace{3pt}}c@{\hspace{3pt}}c
  @{\hspace{3pt}}c@{\hspace{3pt}}c}
& & & & & & & & & & & & & & & 4 & & \\
& & & & & & & & & & & & & & $\swarrow$ & & $\nwarrow$ & \\
$(13\rightarrow)\ 8\!\rightarrow\!5\!\rightarrow\!$ & 3 & $\ldots$
 & $\ldots$ & $\ldots$ & $\ldots$ & $\ldots$ & 2 & $\ldots$
 & $\ldots$ & $\ldots$ & $\ldots$ & $\ldots$  & 1 & & & & 9 \\
& & $\nwarrow$ & & & & $\nearrow$ & & $\nwarrow$ & & &
 & $\nearrow$ & & $\nwarrow$ & & $\nearrow$ & \\
& & & 10 & $\leftarrow$ & 11 & & & & 6 & $\leftarrow$
 & 7 & & & & 12
\end{tabular}
\end{center}
\caption{Ball and chain for 12 or 13.}
\label{fig:ball+chain}
\end{figure}

Although 2 is adjacent to 1 and 3, the chain for 12 or 13
is essentially unique, except that the right tail may be
12 or 4 for either chain.  None of the Fibonacci chains
that we have seen will form a necklace; nor will any others.

The rest of the proof is by induction, but the comparatively
simple pattern is made more difficult to describe by the
fact that only every third Fibonacci number is even.

Balls and chains occur just for $F_{3m+1}-1$ and
$F_{3m+1}$ with $m\geq1$; other cases are simple chains.
The chain 1---2---3 can be thought of as the
``zeroth ball'' (Figure \ref{fig:zeroball+chain}).

\begin{figure}[h]
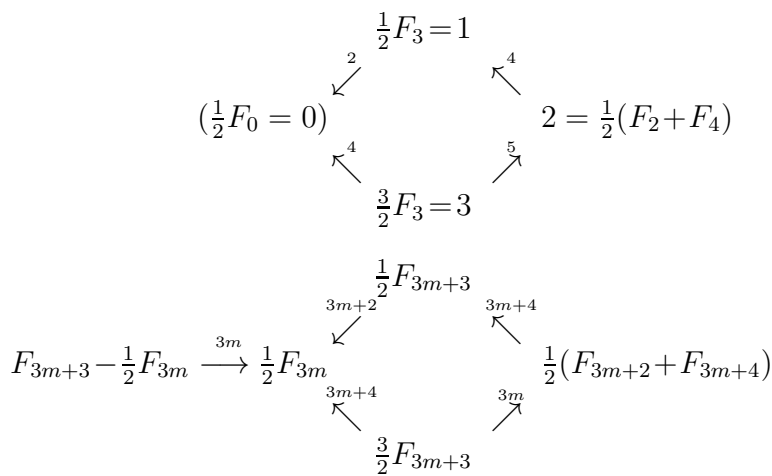
    
\begin{center}
\begin{tabular}{r@{\hspace{-5pt}}c@{\hspace{-2pt}}c
  @{\hspace{1pt}}c@{\hspace{1pt}}l}

 & & ${1\over2}F_3\!=\!1$ & & \\
 & $\stackrel{{\ }_2}{\swarrow}$ &
  & $\stackrel{{\ }_4}{\nwarrow}$ & \\
$({1\over2}F_0=0)$ & & & & $2={1\over2}(F_2\!+\!F_4)$ \\
  & $\stackrel{{\ }_4}{\nwarrow}$ &
   & $\stackrel{{\ }_5}{\nearrow}$ & \\
 & & ${3\over2}F_3\!=\!3$ & & \\[12pt]
   & & ${1\over2}F_{3m+3}$ & & \\
  & $\stackrel{{\ }_{3m+2}}{\swarrow}$ &
  & $\stackrel{{\ }_{3m+4}}{\nwarrow}$ & \\[-6pt]
$F_{3m+3}\!-\!{1\over2}F_{3m}
     \stackrel{{\ }^{3m}}{\longrightarrow}{1\over2}F_{3m}$
     & & & & ${1\over2}(F_{3m+2}\!+\!F_{3m+4})$ \\
    & $\stackrel{{\ }_{3m+4}}{\nwarrow}$ &
   & $\stackrel{{\ }_{3m}}{\nearrow}$ & \\
    & & ${3\over2}F_{3m+3}$ & &

\end{tabular}
\end{center}
\caption{Zeroth ball and general ball.  Small
   numbers above the arrows are ranks of Fibonacci
   numbers to which pairs of linked numbers sum.}
\label{fig:zeroball+chain}
\end {figure}

There are no chains for $n=14$, 15, 16, 17, 18 or 19,
since, when we successively append these numbers to the
graph, the first three are monovalent vertices, as also
is 17 $(={1\over2}F_9)$, though this last can be
accommodated by breaking the ball and allowing 17 to
become an end of the chain.  When we adjoin 18 \& 19
they repectively allow 16 \& 15 to become bivalent,
but a chain is not reachieved until we append $F_8-1=20$
at 1 \& 14.

Note that all the partitions (5\&3, 2\&6, 7\&1) of
$F_6=8$ into two distinct parts have been bypassed by
the partitions of $F_9=34$ into parts of size less
than $F_8=21$, which itself can then be appended
to form a new tail to the chain.  Because $F_9$ is
even, as is every third Fibonacci number,
${1\over2}F_9=17$ can only be appended to
4 $(={1\over2}F_6)$.

\begin{figure}[h]
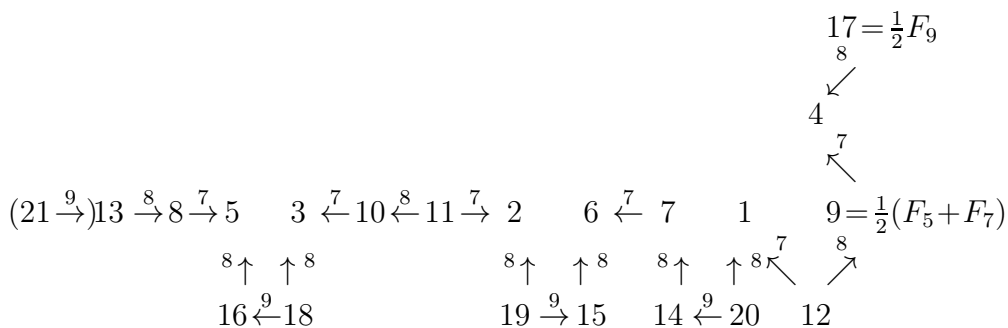
      
\begin{center}
\begin{tabular}{r@{\hspace{-3pt}}c@{\hspace{-1pt}}c
 @{\hspace{-1pt}}c@{\hspace{1pt}}c@{\hspace{1pt}}c
 @{\hspace{1pt}}c@{\hspace{1pt}}c@{\hspace{1pt}}c
 @{\hspace{1pt}}c@{\hspace{1pt}}c@{\hspace{1pt}}c
 @{\hspace{1pt}}c@{\hspace{1pt}}c@{\hspace{-2pt}}c
 @{\hspace{-12pt}}l}
 & & & & & & & & & & & & & & & $17\!=\!{1\over2}F_9$ \\[-2pt]
 & & & & & & & & & & & & & & $\stackrel{8}{\swarrow}$ & \\[-1pt]
 & & & & & & & & & & & & & 4 & & \\[-2pt]
 & & & & & & & & & & & & & & $\stackrel{7}{\nwarrow}$ & \\
$(21\!\stackrel{9}{\rightarrow})\!13\stackrel{8}{\rightarrow}
   \!8\!\stackrel{7}{\rightarrow}$
   & 5 & & 3 & $\stackrel{7}{\leftarrow}\!10
   \!\stackrel{8}{\leftarrow}\!11\!\stackrel{7}{\rightarrow}$
   & 2 & & 6 & $\stackrel{7}{\leftarrow}$ & 7 & & 1
   & & & & $9\!=\!{1\over2}(F_5\!+\!F_7)$ \\
 & ${\ }^8\!\uparrow$ & & $\uparrow\hspace{-4pt}{\ }^8$ &
 & ${\ }^8\!\uparrow$ & & $\uparrow\hspace{-4pt}{\ }^8$ &
 & ${\ }^8\!\uparrow$ & & $\uparrow\hspace{-4pt}{\ }^8$
 & $\stackrel{7}{\nwarrow}$ & & $\stackrel{8}{\nearrow}$ &
\\    & 16 & $\stackrel{9}{\leftarrow}$ & 18 & & 19
 & $\stackrel{9}{\rightarrow}$ & 15 & & 14
 & $\stackrel{9}{\leftarrow}$ & 20 & & 12 & &

\end{tabular}
\end{center}
\caption{Fibonacci chains for $F_8-1=20$ and $F_8=21$.}
\label{fig:FibChain}
\end{figure}

If we continue, we find that a chain cannot again be
achieved until we have replaced the six partitions of
$F_7=13$ by links of partitions of $F_{10}=55$ into
two parts of size at most $F_9-1=33$ (Figure \ref{fig:links}).

\begin{figure}[h]
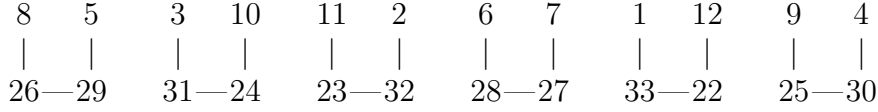
    
\begin{center}
\begin{tabular}{c@{\hspace{1pt}}c@{\hspace{1pt}}c
@{\hspace{21pt}}c@{\hspace{1pt}}c@{\hspace{1pt}}c
@{\hspace{21pt}}c@{\hspace{1pt}}c@{\hspace{1pt}}c
@{\hspace{21pt}}c@{\hspace{1pt}}c@{\hspace{1pt}}c
@{\hspace{21pt}}c@{\hspace{1pt}}c@{\hspace{1pt}}c
@{\hspace{21pt}}c@{\hspace{1pt}}c@{\hspace{1pt}}c}
8 & & 5 & 3 & & 10 & 11 & & 2 & 6 & & 7 & 1 & & 12 & 9 & & 4 \\
$\mid$ & & $\mid$ & $\mid$ & & $\mid$ & $\mid$ & & $\mid$
 & $\mid$ & & $\mid$ & $\mid$ & & $\mid$ & $\mid$ & & $\mid$ \\
26 & --- & 29 & 31 & --- & 24 & 23 & --- & 32 & 28 & --- & 27
 & 33 & --- & 22 & 25 & --- & 30
\end{tabular}
\end{center}
\caption{Links extending the chain to $F_9-1=33$.}
\label{fig:links}
\end{figure}

$F_9=34$ can then be appended to $21=F_8$ to make a new
tail to the chain.

The next chain is for $F_{10}-1=54$, obtained by
appending links of partitions of $F_{11}=89$ into
parts of size at most 54: 54---35, 53---36, $\ldots$,
45---44 to the ten partitions 1---20, 2---19, $\ldots$,
10---11, of $F_8=21$. The chain for $F_{10}=55$ can be
formed by appended it at the end $F_9=34$.

Note that the link ---51---38--- {\bf need not
immediately} replace the end link, \\ ---4---17, of
the chain, but the latter can remain as part of a new
ball, the case $m=2$ of Figure \ref{fig:zeroball+chain}, until we wish to
append ${1\over2}F_{12}=72$, which we will do when
forming the 88- and 89-chains.

We have seen several stages of the induction.
In Figure \ref{fig:ball+chain} the numbers between $F_5=5$ and $F_6=8$
and $F_6$ itself are appended, as also are the numbers
between $F_6=8$ and $F_7=13$ and 13 itself. In Figures
7 and 8, the numbers between $F_k$ and $F_{k+1}$ are
appended for $k=7$ and 8 respectively.  Note that in
the former ${1\over2}F_{k+2}=17$ is appended to
${1\over2}F_{k-1}=4$.

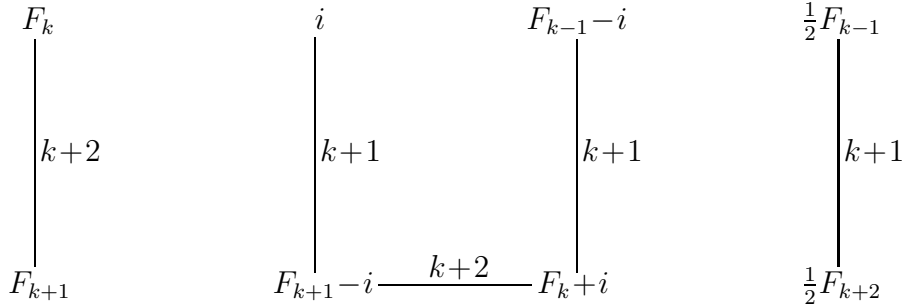
\begin{figure}[h]     
\begin{picture}(350,120)(-30,0)
\put(0,0){$F_{k+1}$}
\put(100,0){$F_{k+1}\!-\!i$}
\put(200,0){$F_k\!+\!i$}
\put(300,0){${1\over2}F_{k+2}$}
\put(5,100){$F_k$}
\put(116,100){$i$}
\put(196,100){$F_{k-1}\!-\!i$}
\put(300,100){${1\over2}F_{k-1}$}
\put(10,10){\line(0,1){86}}
\put(116,8){\line(0,1){89}}
\put(215,8){\line(0,1){88}}
\put(314,10){\line(0,1){86}}
\put(140,3){\line(1,0){58}}
\put(12,50){$k\!+\!2$}
\put(118,50){$k\!+\!1$}
\put(217,50){$k\!+\!1$}
\put(316,50){$k\!+\!1$}
\put(159,6){$k\!+\!2$}
\end{picture}
\medskip
\caption{Extending $F_k\!-\!1$ and $F_k$ chains to
those for $F_{k+1}\!-\!1$ and $F_{k+1}$  The
appendage on the right is required only when $k=3m+1$.}
\label{fig:FibExtension}
\end{figure}

Generally, as in Figure \ref{fig:FibExtension}, we append the pairs of
numbers $F_k+i$, $F_{k+1}-i$ for
$1\leq i\leq {1\over2}(F_{k-1}-1)$, except that, when
$k=3m+1$, ${1\over2}(F_{k-1}-1)$ is not an integer and
we have a new tail, ${1\over2}F_{k+2}$, which {\bf is}
an integer, appended to ${1\over2}F_{k-1}$.
\end{proof}

These last numbers are denominators of the convergents
to the continued fraction for $\sqrt{5}$, sequence
A001076 in Neil Sloane's Online Encyclopedia of
Integer Sequences~\cite{S}.


The proof can be made much more perspicuous with
billiards diagrams, which will also throw light
on the other kinds of chain in which we are
interested.

\begin{figure}[h]  
\begin{picture}(200,120)(-150,-20)
\put(30,0){\line(1,1){75}}
\put(60,0){\line(1,1){60}}
\put(90,0){\line(1,1){30}}
\put(0,30){\line(1,-1){30}}
\put(0,60){\line(1,-1){60}}
\put(15,75){\line(1,-1){75}}
\put(45,75){\line(1,-1){75}}
\put(75,75){\line(1,-1){45}}
\put(105,75){\line(1,-1){15}}
\thicklines
\put(0,0){\line(1,1){75}}
\put(0,30){\line(1,1){45}}
\put(0,60){\line(1,1){15}}
\put(-0.5,0.5){\line(1,1){75}}
\put(-0.5,30.5){\line(1,1){45}}
\put(-0.5,60.5){\line(1,1){15}}
\put(28,-10){1}
\put(57,-10){2}
\put(87,-10){3}
\put(120,-10){4}
\put(123,26){5}
\put(123,56){6}
\put(102,77){7}
\put(72,77){8}
\put(42,77){9}
\put(10,77){10}
\put(-12,57){11}
\put(-12,27){12}
\put(-10,-10){13}
\tiny
\put(118,73){(6${1\over2}$)}
\put(-15,73){(10${1\over2}$)}
\normalsize
\end{picture}
\caption{Fibonacci plays billiards. The thick upward
paths connect 21-sums. The other upward paths connect
8-sums.  The down paths connect 13-sums.}
\label{fig:FibBillards}
\end{figure}
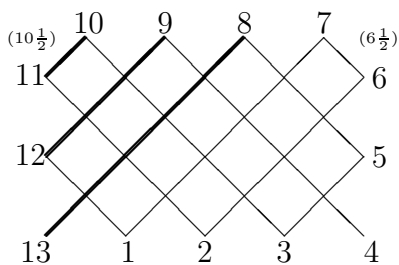

Figure \ref{fig:FibBillards} is equivalent to Figure \ref{fig:ball+chain}.  The `ball' may
be achieved by connecting the Fibonacci sum 1---+---4 = 5.

This billiard table viewpoint is useful for depicting
long chains whose adjacent pair-sums all lie in a set
of only three or four elements.  If successive corners
are at $a$, $b$, $c$, $d$, where $a<b<c<d$, then the
semi-perimeter must be $c-a = d-b$, and the perimeter is
$P = 2(c-a) = 2(d-b)$. One side must be $b-a=d-c$,
and the other must be $c-b=a-d\pmod P$. Viewed
along the 45 degree path taken by the billiard
ball, each integer along the side of the table
has valence 2, and each integer in a corner has
valence 1. Hence, if the corners include 2 integers
(called pockets) and 2 non-integers, then the path
beginning at either pocket must eventually terminate
in the other pocket.

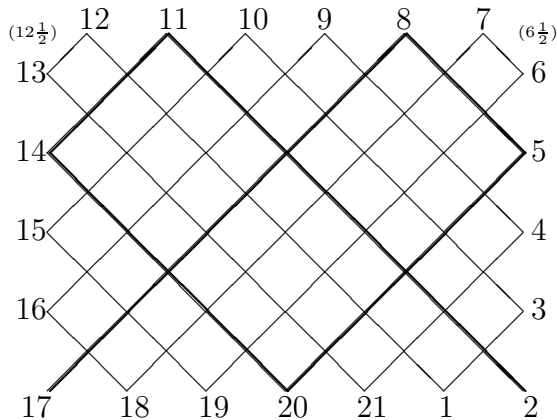
\begin{figure}[h]  
\begin{picture}(200,160)(-100,-5)
\put(30,0){\line(1,1){135}}
\put(60,0){\line(1,1){120}}
\put(90,0){\line(1,1){90}}
\put(120,0){\line(1,1){60}}
\put(150,0){\line(1,1){30}}
\put(0,0){\line(1,1){135}}
\put(0,30){\line(1,1){105}}
\put(0,60){\line(1,1){75}}
\put(0,90){\line(1,1){45}}
\put(0,120){\line(1,1){15}}
\put(0,30){\line(1,-1){30}}
\put(0,60){\line(1,-1){60}}
\put(0,90){\line(1,-1){90}}
\put(0,120){\line(1,-1){120}}
\put(15,135){\line(1,-1){135}}
\put(45,135){\line(1,-1){135}}
\put(75,135){\line(1,-1){105}}
\put(105,135){\line(1,-1){75}}
\put(135,135){\line(1,-1){45}}
\put(165,135){\line(1,-1){15}}
\thicklines
\put(1,0){\line(1,1){135}}
\put(91,0){\line(1,1){90}}
\put(1,90){\line(1,1){45}}
\put(1,90){\line(1,-1){90}}
\put(46,135){\line(1,-1){135}}
\put(136,135){\line(1,-1){45}}
\put(148,-10){1}
\put(180,-10){2}
\put(183,27){3}
\put(183,57){4}
\put(183,87){5}
\put(183,117){6}
\put(162,137){7}
\put(132,137){8}
\put(102,137){9}
\put(72,137){10}
\put(42,137){11}
\put(12,137){12}
\put(117,-10){21}
\put(87,-10){20}
\put(57,-10){19}
\put(27,-10){18}
\put(-10,-10){17}
\put(-12,27){16}
\put(-12,57){15}
\put(-12,87){14}
\put(-12,117){13}
\tiny
\put(178,134){(6${1\over2}$)}
\put(-15,134){(12${1\over2}$)}
\normalsize
\end{picture}
\medskip
\caption{A billiard table with $A=4$, $B=13$, $C=25$,
$D=34$ and perimeter $P=21$.  The double-sides
$B-A=9$ and $C-B=12$ are not relatively prime.}
\label{fig:BilliardTable}
\end{figure}

Figure \ref{fig:BilliardTable} shows a rectangle of perimeter 21, whose
corners are at $a=2$, $b=6.5$, $c=12.5$, $d=17$.
The sequence between pockets (thick lines) is
2, 11, 14, 20, 5, 8, 17. This sequence fails to
reach many of the other integers along the perimeter,
which lie in the following cycle:
1, 3, 10, 15, 19, 6, 7, 18, 16, 9, 4, 21, 13, 12, 1.
The question of which rectangular billiard tables yield
a single covering path and which yield a degeneracy
of this sort is answered by the following lemma.

\begin{lemma}
Let $A$, $B$, $C$, $D$, be positive integers
such that $A < B < C < D$ and $C-A = D-B$. Let
$a = A/2$; $b = B/2$; $c = C/2$, and $d = D/2$.
Further suppose that exactly two of $a$, $b$, $c$, $d$
are integers, so that the corresponding billiard table
has two corner pockets. Then the 45 degree path
between the pockets touches {\bf all} of the integers
along the perimeter just if the rectangle's
double-sides, $B-A$ and $C-B$, are relatively prime.
\end{lemma}

\begin{note}
In Figure \ref{fig:BilliardTable} the sides are $6.5-2=4.5$ and
$12.5-6.5=6$, so the double-sides are 9 and 12.
They have a common factor of 3. So we could color
each integer of shape $3k+2$ and both pockets
would be colored. Every integer along the ball's
path would then also be colored. In general,
this argument shows that a degeneracy occurs
whenever the double-sides are not relatively prime.
\end{note}


\noindent {\bf Proof of non-degeneracy}. If the double-sides
are prime to each other, and hence to the perimeter
$P=C-A=D-B,$ so that, mod $P$, $A\equiv C$ and
$D\equiv B$, then consider any two integers
separated by exactly one bounce along the ball's
path. If the bounce is at $x$, these integers,
mod $P$, are at $A-x$ and $B-x$, and the distance
between them is $B-A\equiv D-C$ if measured in one
direction mod $P$, or
$A-B\equiv C-D\equiv A-D\equiv C-B$ if measured in
the other direction. But since $B-A$ is a
double-side, which is relatively prime to $P$,
it follows that the sequence, obtained by looking
at {\bf alternate} bounce-points along the ball's
path, cannot cycle back to itself, mod $P$, without
first reaching a pocket. Since this is true for
all values of $x$, the ball-path from one pocket to
the other must go through every integer point on
the rectangle's perimeter.


We can take three corners of a rectangle as
the halves of any three consecutive Fibonacci
numbers (recall that the corners are allowed to
be half-integers). The perimeter of this
rectangle will be the middle of these three
Fibonacci numbers.  Since any pair of adjacent
Fibonacci numbers is relatively prime, the path
from pocket to pocket is complete.

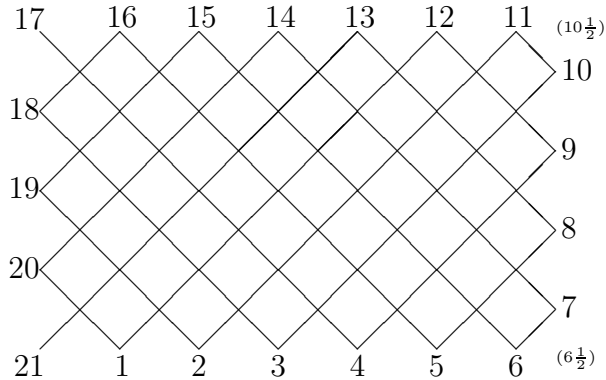
\begin{figure}[h]  
\begin{picture}(220,150)(-100,-10)
\put(30,0){\line(1,1){120}}
\put(60,0){\line(1,1){120}}
\put(90,0){\line(1,1){105}}
\put(120,0){\line(1,1){75}}
\put(150,0){\line(1,1){45}}
\put(180,0){\line(1,1){15}}
\put(75,75){\line(1,1){45}}
\put(105,75){\line(1,1){15}}
\put(0,0){\line(1,1){120}}
\put(0,30){\line(1,1){90}}
\put(0,60){\line(1,1){60}}
\put(0,90){\line(1,1){30}}
\put(0,30){\line(1,-1){30}}
\put(0,60){\line(1,-1){60}}
\put(0,90){\line(1,-1){90}}
\put(0,120){\line(1,-1){120}}
\put(30,120){\line(1,-1){120}}
\put(60,120){\line(1,-1){120}}
\put(90,120){\line(1,-1){105}}
\put(120,120){\line(1,-1){75}}
\put(150,120){\line(1,-1){45}}
\put(180,120){\line(1,-1){15}}
\put(28,-10){1}
\put(57,-10){2}
\put(87,-10){3}
\put(117,-10){4}
\put(147,-10){5}
\put(177,-10){6}
\put(197,12){7}
\put(197,42){8}
\put(197,72){9}
\put(197,102){10}
\put(175,122){11}
\put(145,122){12}
\put(115,122){13}
\put(85,122){14}
\put(55,122){15}
\put(25,122){16}
\put(-10,121){17}
\put(-12,87){18}
\put(-12,57){19}
\put(-12,27){20}
\put(-10,-10){21}
\tiny
\put(195,-5){(6${1\over2}$)}
\put(195,120){(10${1\over2}$)}
\normalsize
\end{picture}
\caption{A billiard table giving a Fibonacci chain of
length $P=21$.}
\label{fig:BilliardTable2}
\end{figure}

\textbf{Square chains.}  For the `square' chains and
necklaces which we mentioned at the outset, Ed Pegg and Edwin Clark have verified that there are
chains for $n=15$, 16, 17, 23, 25 to 31 and necklaces (and hence chains) for $n=32$ upwards. The
existence problem was solved quite recently; more in the appendix at the end.

The billiards technique allows us to construct arbitrarily large specimens. Figure
\ref{fig:BilliardTable3} shows how our billiard table technique can be used to find a `square'
chain of length 16.


\begin{figure}[h]  
\begin{picture}(150,120)(-130,0)
\put(30,0){\line(1,1){105}}
\put(60,0){\line(1,1){75}}
\put(90,0){\line(1,1){45}}
\put(120,0){\line(1,1){15}}
\put(0,0){\line(1,1){105}}
\put(0,30){\line(1,1){75}}
\put(0,60){\line(1,1){45}}
\put(0,90){\line(1,1){15}}
\put(0,30){\line(1,-1){30}}
\put(0,60){\line(1,-1){60}}
\put(0,90){\line(1,-1){90}}
\put(15,105){\line(1,-1){105}}
\put(45,105){\line(1,-1){90}}
\put(75,105){\line(1,-1){60}}
\put(105,105){\line(1,-1){30}}
\put(28,-10){1}
\put(57,-10){2}
\put(87,-10){3}
\put(117,-10){4}
\put(137,12){5}
\put(137,42){6}
\put(137,72){7}
\put(136,105){8}
\put(102,108){9}
\put(72,108){10}
\put(42,108){11}
\put(12,108){12}
\put(-14,87){13}
\put(-14,57){14}
\put(-14,27){15}
\put(-11,-8){16}
\tiny
\put(132,-5){(4${1\over2}$)}
\put(-15,103){(12${1\over2}$)}
\end{picture}
\bigskip
\caption{A billiard table with $A=9$, $B=16$, $C=25$
(all squares), and $D=32$ (half a square) and
perimeter $P=16$.}
\label{fig:BilliardTable3}
\end{figure}
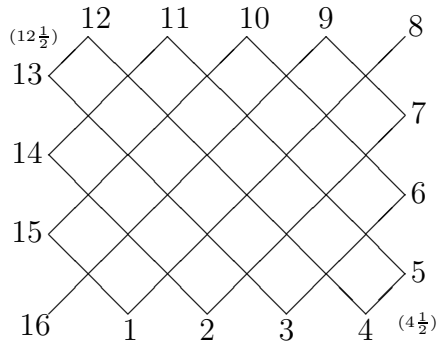

We may delete 16, or append 17, giving the 15-,
16- and 17-chains of Figure \ref{fig:RecamanProblem}.

It is possible to accommodate other numbers by using
billiard tables with more than four corners!
Figure \ref{fig:BilliardTable4} shows such a table with corners at
4.5, 8.5, 9, 12.5, 24.5, and 32. The
corner at 8.5 is reflex; the others are right.
The perimeter is 39. There are two pockets: a
conventional corner pocket at 32, and a
side pocket at 9. The path
between these two pockets is complete.

\begin{figure}[h]  
\begin{picture}(350,240)(-20,0)
\put(30,0){\line(1,1){225}}
\put(60,0){\line(1,1){225}}
\put(90,0){\line(1,1){225}}
\put(120,0){\line(1,1){225}}
\put(150,0){\line(1,1){210}}
\put(180,0){\line(1,1){180}}
\put(210,0){\line(1,1){150}}
\put(240,0){\line(1,1){105}}
\put(270,0){\line(1,1){75}}
\put(300,0){\line(1,1){45}}
\put(330,0){\line(1,1){15}}
\put(0,0){\line(1,1){225}}
\put(0,30){\line(1,1){195}}
\put(0,60){\line(1,1){165}}
\put(0,90){\line(1,1){135}}
\put(0,120){\line(1,1){105}}
\put(0,150){\line(1,1){75}}
\put(0,180){\line(1,1){45}}
\put(0,210){\line(1,1){15}}
\put(0,30){\line(1,-1){30}}
\put(0,60){\line(1,-1){60}}
\put(0,90){\line(1,-1){90}}
\put(0,120){\line(1,-1){120}}
\put(0,150){\line(1,-1){150}}
\put(0,180){\line(1,-1){180}}
\put(0,210){\line(1,-1){210}}
\put(15,225){\line(1,-1){225}}
\put(45,225){\line(1,-1){225}}
\put(75,225){\line(1,-1){225}}
\put(105,225){\line(1,-1){225}}
\put(135,225){\line(1,-1){210}}
\put(165,225){\line(1,-1){180}}
\put(195,225){\line(1,-1){150}}
\put(225,225){\line(1,-1){120}}
\put(255,225){\line(1,-1){105}}
\put(285,225){\line(1,-1){75}}
\put(315,225){\line(1,-1){45}}
\put(345,225){\line(1,-1){15}}
\put(237,-10){1}
\put(267,-10){2}
\put(297,-10){3}
\put(327,-10){4}
\put(347,12){5}
\put(347,42){6}
\put(347,72){7}
\put(347,102){8}
\put(362,117){9}
\put(362,147){10}
\put(362,177){11}
\put(362,207){12}
\put(338,227){13}
\put(308,227){14}
\put(278,227){15}
\put(248,227){16}
\put(218,227){17}
\put(188,227){18}
\put(158,227){19}
\put(128,227){20}
\put(98,227){21}
\put(68,227){22}
\put(38,227){23}
\put(8,227){24}
\put(-12,206){25}
\put(-12,176){26}
\put(-12,146){27}
\put(-12,116){28}
\put(-12,86){29}
\put(-12,56){30}
\put(-12,26){31}
\put(-12,-10){32}
\put(23,-10){33}
\put(53,-10){34}
\put(83,-10){35}
\put(113,-10){36}
\put(143,-10){37}
\put(173,-10){38}
\put(203,-10){39}
\tiny
\put(343,-5){(4${1\over2}$)}
\put(358,223){(12${1\over2}$)}
\put(343,116){(8${1\over2}$)}
\put(-10,223){(24${1\over2}$)}
\end{picture}
\bigskip
\caption{A billiard table with six corners.}
\label{fig:BilliardTable4}
\end{figure}
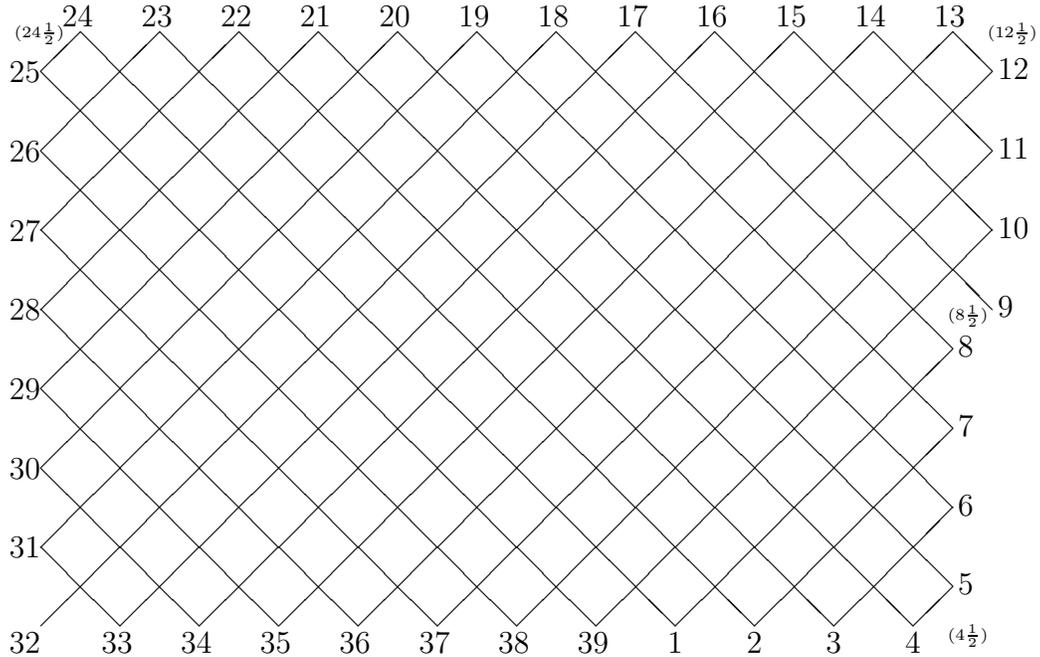

\textbf{Square necklaces.}  In order to connect the
two pockets and make a necklace, we must be sure
that they sum to a square.  Two half-squares
summing to a square are provided by the parametric
equation
$$\left((r+s)^2-2r^2\right)^2+\left((r+s)^2-2s^2\right)^2
=2\left(r^2+s^s\right)^2$$
For example, $1^2+7^2=2\cdot5^2$.  We multiply the
solution by 6 to get the parity right and to avoid the
sides having a common factor of 3. $42^2-6^2=2^6\cdot3^3$
can be arranged as the difference of two odd squares,
which are not multiples of 3, in just two different ways,
$43^2-11^2$ and $433^2-431^2$.  Billiard tables with
half these squares as corners have perimeters 1728 and
185725.  Their double-sides, $(5\cdot17,31\cdot53)$ and
$(2^6\cdot3^3,11\cdot43\cdot389)$ are coprime, so the
chains contain every integer on the perimeter.  Moreover,
the ends of the chains are ${1\over2}6^2$ and
${1\over2}42^2)$ which sum to $30^2$
so that they may be joined to form necklaces.

Here are some small square necklaces.  The bold numbers
are $6x$, $6y$.

\begin{center}
\begin{tabular}{|c|c|c|c|c|} \hline
   & $x^2+y^2=2z^2$ & corners are half      & double sides
   & perimeter \\
$r,s$ &    $x,y$    &  the squares of:      & are coprime
   &  $P$ \\ \hline
2,1 & 1,7   & {\bf6},11,{\bf42},43          & 85,1643
   & 1728 \\
3,2 & 7,17  & {\bf42},{\bf102},119,151      & 3757,8640
   & 12397 \\
4,3 & 17,31 & {\bf102},{\bf186},197,251     & 4213,24192
   & 28405 \\
7,3 & 1,41  & {\bf6},23,{\bf246},247        & 493,59987
   & 60480 \\
7,5 & 23,47 & 109,{\bf138},269,{\bf282}     & 7163,53317
   & 60480 \\
5,4 & 31,49 & {\bf186}.{\bf294},373,437     & 51840,52693
   & 104533 \\
7,3 & 1,41  & {\bf6},{\bf246},397,467       & 60480,97093
   & 157573 \\
2,1 & 1,7   & {\bf6},{\bf42},431,433        & 1728,183997
   & 185725 \\
5,3 & 7,23  & {\bf42},{\bf138},859,869      & 17280,718837
   & 736117 \\
7,3 & 1,41  & {\bf6},{\bf246},2153,2167     & 60480,4574893
   & 4635373 \\
3,2 & 7,17  & {\bf42},{\bf102},2159,2161    & 8640,4650877
   & 4659517 \\
5,4 & 31,49 & {\bf186},{\bf294},2587,2597   & 51840,6606133
   & 6657973 \\
5,3 & 7,23  & {\bf42},{\bf138},4319,4321    & 17280,18634717
   & 18651997 \\
4,3 & 17,31 & {\bf102},{\bf186},6047,6049   & 24192,36531613
   & 36555805 \\
7,5 & 23,47 & {\bf138},{\bf282},15119,15121 & 60480,228504637
   & 228565117 \\

\hline
\end{tabular}
\end{center}

Of course, if one looked for square chains by
putting halves of odd squares at the corners of
a billiard table, then, by Theorem 0 of number
theory, namely that odd squares are congruent to 1 mod 8,
we would find that our tour broke up into
four separate loops, those containing 0 and 1,
$-1$ and 2, $-2$ and 3, and those containing
$-3$ and 4 modulo 8.  However, we are
able to make a single necklace, by breaking the
loops at places which sum to a square on other
loops.  For example, the billiard table with
corners at 4.5, 24.5, 40.5 and 60.5 yields
four 18-loops which may be connected to form a
72-necklace as follows


\begin{center}
$\ldots$ 1 --- 3 $\ldots$ 6 --- 10
 $\ldots$ 71 --- 29 $\ldots$ 52 ---48 $\ldots$
\end{center}

where the dots represent the other 16 members of
each of the four loops.

More generally, if the odd squares are $(s-2r)^2$,
$(s+2r)^2$, $(2s-r)^2$ and $(2s+r)^2$, we will
have $n=3(s^2-r^2)$. In order that the point 1
is on an edge adjacent to the smallest square,
we must have $s\geq r+\sqrt{(9r^2-1)/2}$.

\textbf{Cubic chains.}  The billiard table with
corners at \{62.5, 171.5, 256, 365\} has
perimeter 387.  The sides are relatively prime,
so the path between the pockets is complete.
The adjacent pair-sums are 125, 343, 512 and 730.
In pursuit of a chain all of whose pair-sums
are cubes, we move the corner from 365 to 364.5,
and insert a new reflex corner at 386.5 and a
side pocket at 387. A detailed calculation
reveals that the path between the pockets at
387 and 256 is complete, so we then have a
cubic chain among the numbers from 1 through 387.
This chain uses only the cubes 125, 343, 512
and 729.

By deleting the endpoint at 387 we obtain a cubic
chain among the numbers from 1 through 386. Since
each of our Fibonacci chains also has a pocket at
its highest number, we can similarly delete that
maximum number and obtain a Fibonacci chain among
the numbers from 1 to $F_k -1$, for any $k>3$.
We leave the reader to design billiard tables with
extra corners to accommodate such numbers.

No doubt, in answer to Nob Yoshigara's question,
cubic chains and necklaces exist for all
sufficiently large $n$, but not for $n<295$.
When $n=295$ the graph has just two monovalent
vertices, at 216 and 256, which have to form the
tails of a chain, but it cannot be completed.
We can construct a {\bf cubic necklace} if we
can find a number which is the sum of two odd
cubes in two different ways.  If the cubes are
$a^3 + d^3 = b^3 + c^3$, then we also need that
$a^3<c^3-b^3$ (to make sure the necklace includes
all the numbers from 1 on) and that
$\gcd(c^3-b^3,b^3-a^3)=2$ (else the necklace will
split up into smaller necklaces).  The smallest try is
$23^3+163^3=121^3+137^3$, but the relevant gcd is 14
and we have 7 small necklaces each of length 114256
instead of a single necklace of length 799792.
Fortunately, Andrew Bremner observes that
$21^3+257^3=167^3+231^3$ where
$167^3-21^3=2\cdot13\cdot31\cdot73\cdot79$ and
$231^3-167^2=2^6\cdot119827$ have gcd 2, so that if we
put halves of these four odd cubes at the corners of
a billiard table, we will have a cubic necklace of
length the latter number, 7668928.  Surely there
are smaller ones.

\textbf{Triangular chains} exist for $n=2$ and probably
for all $n\geq9$.  Necklaces appear to exist for
$n\geq12$, except for $n=14$.  We would like to
see proofs of these statements, which we have verified
to $n=70$.  It is easy to find arbitrarily large
triangular chains, by taking numbers which are the
sum of two triangular numbers in two different ways.
If the triangular numbers $A<B<C<D$ are odd and not
all multiples of three (in fact two will have to be
multiples of 3 and two of them congruent to 1 mod 3),
then, by placing their halves at the corners of a
billiard table, we will have a {\bf triangular necklace}
of length $C-A$, provided that the sides of the
table are coprime, and that $A<C-B$ (else we will
lose some of the beads from the beginning of the
necklace).

Here are some triangular necklaces.

\begin{center}
\begin{tabular}{|c|c|c|c|c|} \hline
 corners are half the & sides are &  perimeter; \\
 triangular numbers:  &  coprime  & \# of beads \\ \hline
1, 15, 91, 105        & 7, 38     &   90 \\
55, 153, 253, 351     & 49, 50    &  198 \\
91, 231, 325, 465     & 47, 70    &  234 \\
15, 253, 465, 703     & 106, 119  &  450 \\
21, 55, 561, 595      & 17, 253   &  540 \\
45, 153, 595, 703     & 54, 221   &  550 \\
91, 253, 741, 903     & 81, 244   &  650 \\
253, 703, 1035, 1485  & 166, 225  &  782 \\
3, 325, 903, 1225     & 161, 289  &  900 \\
325. 703, 1275, 1653  & 189, 286  &  950 \\
45, 91, 1035, 1081    & 23, 472   &  990 \\
465, 703, 1653, 1891  & 119, 475  & 1188 \\
171, 1225, 1431, 2485 & 103, 527  & 1260 \\
45, 325, 1431, 1711   & 140, 553  & 1386 \\
1, 55, 1431, 1485     & 27, 688   & 1430 \\
45, 1035, 1711, 2701  & 338, 495  & 1666 \\
1, 435, 1711, 2145    & 217, 638  & 1710 \\
171, 703, 1953, 2485  & 266, 625  & 1782 \\
91, 153, 1891, 1953   & 31, 869   & 1800 \\
55, 1485, 1891, 3321  & 203, 715  & 1836 \\
105, 595, 2211, 2701  & 245, 808  & 2106 \\
15, 231, 2485, 2701   & 108, 1127 & 2470 \\
91, 1485, 2701, 4095  & 608, 697  & 2610 \\
55, 435, 2701, 3081   & 190, 1133 & 2646 \\
21, 595, 3081, 3655   & 287, 1243 & 3060 \\
3, 325, 3081, 3403    & 161, 1378 & 3078 \\
171, 253, 3321, 3403  & 41, 1584  & 3250 \\
1, 91, 4005, 4095     & 45, 1957  & 4004 \\
\hline
\end{tabular}
\end{center}

The existence of `triangular triples', such as
---29---91---62---, ---44---92---61---, ---27---93---78---
in which each pair sums to a triangular number, enable
us to expand the 90-necklace at the head of the last list,
to 91-, 92- and 93-necklaces, as in Figure \ref{fig:NecklaceExtension}.

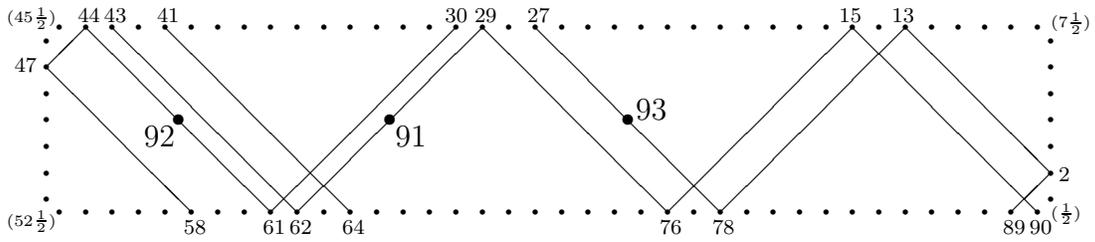
\begin{figure}[h]  
\begin{picture}(400,100)(0,0)
\put(5,55){\line(1,1){15}}
\put(90,0){\line(1,1){70}}
\put(100,0){\line(1,1){70}}
\put(240,0){\line(1,1){70}}
\put(260,0){\line(1,1){70}}
\put(370,0){\line(1,1){15}}
\put(5,55){\line(1,-1){55}}
\put(20,70){\line(1,-1){70}}
\put(30,70){\line(1,-1){70}}
\put(50,70){\line(1,-1){70}}
\put(170,70){\line(1,-1){70}}
\put(190,70){\line(1,-1){70}}
\put(310,70){\line(1,-1){70}}
\put(330,70){\line(1,-1){55}}
\put(388,12){{\scriptsize 2}}
\put(377,-8){{\scriptsize 90}}
\put(367,-8){{\scriptsize 89}}
\put(325,72){{\scriptsize 13}}
\put(305,72){{\scriptsize 15}}
\put(257,-8){{\scriptsize 78}}
\put(237,-8){{\scriptsize 76}}
\put(187,72){{\scriptsize 27}}
\put(167,72){{\scriptsize 29}}
\put(156,72){{\scriptsize 30}}
\put(117,-8){{\scriptsize 64}}
\put(97,-8){{\scriptsize 62}}
\put(87,-8){{\scriptsize 61}}
\put(47,72){{\scriptsize 41}}
\put(57,-8){{\scriptsize 58}}
\put(27,72){{\scriptsize 43}}
\put(17,72){{\scriptsize 44}}
\put(-7,53){{\scriptsize 47}}
\multiput(10,0)(10,0){38}{\circle*{2}}
\multiput(10,70)(10,0){38}{\circle*{2}}
\multiput(5,5)(0,10){7}{\circle*{2}}
\multiput(385,5)(0,10){7}{\circle*{2}}
\put(55,35){\circle*{4}}
\put(135,35){\circle*{4}}
\put(225,35){\circle*{4}}
\put(42,25){92}
\put(137,25){91}
\put(228,35){93}
\tiny
\put(-10,-5){(52${1\over2}$)}
\put(-10,72){(45${1\over2}$)}
\put(385,-2){(${1\over2}$)}
\put(385,70){(7${1\over2}$)}
\end{picture}
\medskip
\caption{Expanding a `triangular' 90-necklace by one,
two or three beads.}
\label{fig:NecklaceExtension}
\end{figure}

In the same way we can insert ---101---199---152---
and ---100---200---53--- into the 198-necklace which
is the second in the list.

\textbf{Pentagonal chains}, i.e., those in which
adjacent links sum to the pentagonal numbers,
1, 2, 5, 7, 12, 15, $\ldots$, ${1\over2}n(3n\pm1)$,
appear to exist for all $n\geq4$ (e.g., 1---4---3---2)
and necklaces for all $n\geq9$.  E.g.,
 --6--1--4--8--7--5--2--3--9--6-- or

\normalfont\begin{figure}[h]
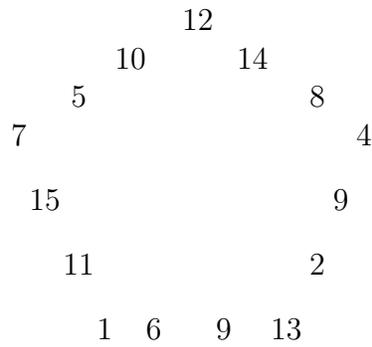
       
\begin{center}
\begin{tabular}{c@{\hspace{1pt}}c@{\hspace{1pt}}c
 @{\hspace{1pt}}c@{\hspace{1pt}}c@{\hspace{0pt}}c
 @{\hspace{0pt}}c@{\hspace{1pt}}c@{\hspace{2pt}}c
 @{\hspace{1pt}}c@{\hspace{3pt}}c@{\hspace{3pt}}c
 @{\hspace{3pt}}c}
 & & & & & & \ \ 12 & & & & & & \\
 & & & & 10 & & & & 14 & & & & \\
 & & 5 & & & & & & & & 8 & & \\
7 & & & & & & & & & & & & 4 \\[10pt]
 & 15 & & & & & & & & & & 9 & \\[10pt]
 & & 11 & & & & & & & & 2 & & \\[10pt]
 & & & 1 & & 6 & & 9 & & 13 & & &
\end{tabular}
\end{center}
\caption{A necklace with adjacent pairs of beads adding to
  pentagonal numbers.}
	\label{fig:LastNecklace}
\end{figure}

This has been checked to $n=49$.  Here are some other
necklaces.

\begin{center}
\begin{tabular}{|c|c|c|c|c|} \hline
 corners are half the & sides are &  perimeter; \\
 pentagonal numbers:  &  coprime  & \# of beads \\ \hline
15, 35, 57, 77    & 10, 11  &  42 \\
1, 7, 51, 57      & 3, 22   &  50 \\
7, 35, 117, 145   & 14, 41  & 110 \\
35, 77, 145, 187  & 21, 34  & 110 \\
15, 117, 145, 247 & 51, 14  & 130 \\
57, 155, 247, 345 & 49, 46  & 190 \\
7, 145, 287, 425  & 69, 71  & 280 \\
1, 15, 287, 301   & 7, 136  & 286 \\
7, 51, 301, 345   & 22, 125 & 294 \\
7. 77, 425, 495   & 35, 174 & 418 \\
\hline
\end{tabular}
\end{center}


\textbf{Prime chains} have been considered from time to
time \cite{G,K}, but as in all cases except the
Fibonacci numbers and the Lucas numbers, existence
proofs for all large enough $n$ are elusive.

\begin{theorem} \label{th:2}
There is a chain formed with the
numbers 1 to $n$ with each adjacent pair adding to
a Lucas number, just if $n=5$, or $L_k$ or
$L_k - 1$, where $L_k$ is a Lucas number with $k\geq2$
($L_2=3$, $L_3=4$, $L_{n+1}=L_n+L_{n-1}$).  The chain
is essentially unique.
\end{theorem}

The proof can follow either of the methods used for
Theorem 1.

There are corresponding theorems for sequences
satisfying the same recurrence relation.  For example,
the chains that can be formed using the numbers
4, 5, 9, 14, 23, 37, $\ldots$ have length
one of those numbers, or one less than one of them.

\subsubsection*{Appendix on square necklaces.}
In the sixteen years since this paper was written, one author has collected square necklaces for
$32\leq n\leq 252$. They are not unique. Figure~\ref{fig:square40} shows a pair of necklaces for
$n=40$.

\begin{figure}[h]       
\begin{center}
\begin{tabular}{c@{\hspace{4pt}}c@{\hspace{4pt}}c
 @{\hspace{4pt}}c@{\hspace{4pt}}c@{\hspace{4pt}}c
 @{\hspace{4pt}}c@{\hspace{4pt}}c@{\hspace{4pt}}c
@{\hspace{4pt}}c@{\hspace{4pt}}c@{\hspace{4pt}}cc
@{\hspace{4pt}}c@{\hspace{4pt}}c@{\hspace{4pt}}c
@{\hspace{4pt}}c@{\hspace{4pt}}c@{\hspace{4pt}}c
@{\hspace{4pt}}c@{\hspace{4pt}}c@{\hspace{4pt}}c
@{\hspace{4pt}}c@{\hspace{4pt}}c}
 1 & 3  & 6  & 19 & 30 & 34 & 15 & 10 & 39 & 25 & 24 & \qquad \qquad &  3 & 6  & 19 & 30 & 34 & 15 & 10 & 39 & 25 & 24 & 40 \\
 8 &    &    &    &    &    &    &    &    &    & 40 &               & 22 &    &    &    &    &    &    &    &    &    &  9 \\
17 &    &    &    &    &    &    &    &    &    &  9 &               & 27 &    &    &    &    &    &    &    &    &    & 16 \\
32 &    &    &    &    &    &    &    &    &    & 16 &               & 37 &    &    &    &    &    &    &    &    &    & 33 \\
 4 &    &    &    &    &    &    &    &    &    & 33 &               & 12 &    &    &    &    &    &    &    &    &    & 31 \\
21 &    &    &    &    &    &    &    &    &    & 31 &               & 13 &    &    &    &    &    &    &    &    &    & 18 \\
28 &    &    &    &    &    &    &    &    &    & 18 &               & 36 &    &    &    &    &    &    &    &    &    &  7 \\
36 &    &    &    &    &    &    &    &    &    &  7 &               & 28 &    &    &    &    &    &    &    &    &    & 29 \\
13 &    &    &    &    &    &    &    &    &    &  2 &               & 21 &    &    &    &    &    &    &    &    &    & 20 \\
12 &    &    &    &    &    &    &    &    &    & 23 &               &  4 &    &    &    &    &    &    &    &    &    &  5 \\
37 & 27 & 22 & 14 & 35 & 29 & 20 &  5 & 11 & 38 & 26 & \qquad \qquad & 32 & 17 &  8 &  1 & 35 & 14 &  2 & 23 & 26 & 38 & 11
  \end{tabular}
\end{center}
\caption{A pair of square necklaces for $n=40$.}
\label{fig:square40}
\end{figure}

At a recent MathFest presentation by the other author, a member of the audience claimed to have
used a computer to find square necklaces for $32\leq n\leq 1000$.

We were delighted to learn that the problem was recently solved by Robert Gerbicz; see the Mersenne
Forum blog thread \cite{Ge1}. Square necklaces exist for any length of the form $n = (71*25^k -
1)/2$ with $k \geq 0$. A generalization of this construction proves the existence of square
necklaces of any length $n \geq 32$ and square chains of any length $n \geq 25$. Gerbicz's C code
for generating square necklaces is available for download~\cite{Ge2}.

\textbf{Acknowledgment.}
Thanks to Alex Fink for finding one of the necklaces in Figure \ref{fig:square40} and to Ethan White 
for discovering Robert Gerbicz's blog post.

%

\bigskip
\footnoterule

\footnotesize{MSC: Primary 11B75}

\end{document}